\documentclass[british]{article}
\usepackage[T1]{fontenc}
\usepackage[latin9]{inputenc}
\usepackage{color}
\usepackage{babel}
\usepackage{units}
\usepackage{amsthm}
\usepackage{amsmath}
\usepackage{amssymb}
\usepackage{esint}
\inputencoding{utf8}
\usepackage[unicode=true,
 bookmarks=true,bookmarksnumbered=false,bookmarksopen=false,
 breaklinks=false,pdfborder={0 0 0},backref=false,colorlinks=true]
 {hyperref}
\hypersetup{pdftitle={A topological constraint for monotone Lagrangians in hypersurfaces of Kähler manifolds},
 pdfauthor={Simon Schatz},
 pdfsubject={Symplectic topology},
 pdfkeywords={Lifted Floer Homology, Monotone Lagrangians},
 citecolor=blue, urlcolor=cyan}

\makeatletter
\theoremstyle{plain}
\newtheorem{thm}{\protect\theoremname}[section]
  \theoremstyle{plain}
  \newtheorem{cor}[thm]{\protect\corollaryname}
  \theoremstyle{definition}
  \newtheorem{defn}[thm]{\protect\definitionname}
  \theoremstyle{remark}
  \newtheorem{rem}[thm]{\protect\remarkname}
  \theoremstyle{definition}
  \newtheorem{example}[thm]{\protect\examplename}
  \theoremstyle{plain}
  \newtheorem{prop}[thm]{\protect\propositionname}
  \theoremstyle{plain}
  \newtheorem{lem}[thm]{\protect\lemmaname}
  \theoremstyle{plain}
  \newtheorem*{cor*}{\protect\corollaryname}

\makeatother

  \providecommand{\corollaryname}{Corollary}
  \providecommand{\definitionname}{Definition}
  \providecommand{\examplename}{Example}
  \providecommand{\lemmaname}{Lemma}
  \providecommand{\propositionname}{Proposition}
  \providecommand{\remarkname}{Remark}
\providecommand{\theoremname}{Theorem}

\begin{document}

\title{A topological constraint for monotone Lagrangians in hypersurfaces
of Kähler manifolds}

\author{Simon Schatz\\
\href{mailto:sschatz@unistra.fr}{sschatz@unistra.fr}}

\maketitle

\begin{abstract}
In this paper we establish a topological constraint for monotone Lagrangian
embeddings in certain complex hypersurfaces of integral Kähler manifolds.
As an application, we prove that it is impossible to embed a connected
sum of $S^{1}\times S^{2k}$s in $\mathbb{C}P^{2k+1}$ as a monotone
Lagrangian.
\end{abstract}

\section{Introduction and main result}

This paper is concerned with a topological constraint on certain monotone
Lagrangian submanifolds in symplectic hypersurfaces of Kähler manifolds,
with the particular example of $\mathbb{C}P^{n}$ in mind. It is known
that the existence of Lagrangian embeddings $L\hookrightarrow M$
imposes topological constraints on $L$; one may think for instance
of Gromov's celebrated theorem regarding the impossibility for a Lagrangian
submanifold in $\mathbb{C}^{n}$ to be simply connected~\cite{Gromov},
and more recently S. Nemirovski proved that Klein bottles do not admit
a Lagrangian embedding in $\mathbb{C}^{2n}$~\cite{LagKleinBottles}.
We will consider, as in the two aforementioned results, a closed and
connected Lagrangian. In our case, we find that under the right geometrical
circumstances a $K\left(\pi,1\right)$ monotone, orientable Lagrangian
must have some non-trivial element $g\in\pi$ whose centraliser is
of finite index.

This finding actually echoes and generalises in the case of monotone
Lagrangians a claim by Fukaya in~\cite{Fuk}:
\begin{quotation}
Let $L$ be a $K\left(\pi,1\right)$, spin, Lagrangian submanifold
of $\mathbb{C}P^{n}$. Then there is some $A\in\pi_{2}\left(\mathbb{C}P^{n},L\right)$
of Maslov index $2$ such that the centraliser of $\partial A$ is
of finite index in $\pi_{1}\left(L\right)$.
\end{quotation}
We will generalise this statement to the framework developed by P.
Biran in~\cite{Biran:2001fk}, of which $\mathbb{C}P^{n}\subset\mathbb{C}P^{n+1}$
is an example.

\newpage{}
\begin{thm}
\label{prop:Main-result}Let $L^{n}$ be a monotone, compact, orientable
and $K\left(\pi,1\right)$ Lagrangian submanifold of some symplectic
manifold $\left(\Sigma^{2n},\omega\right)$.

Assume that $\Sigma$ is a complex hypersurface of a closed, integral
Kähler manifold $\left(M^{2n+2},\omega_{M}\right)$, that $\left[\Sigma\right]\in H_{2n}\left(M;\mathbb{Z}\right)$
is Poincaré-dual to a multiple of $\left[\omega_{M}\right]\in H^{2}\left(M,\mathbb{Z}\right)$,
that $W=M\setminus\Sigma$ is a subcritical Weinstein domain, and
that the first Chern number of $\Sigma$ is at least $2$.

Then the Maslov number $N_{L}$ of $L$ is $2$ and there exists some
non-trivial $g\in\pi_{1}\left(L\right)$ such that its centraliser
is of finite index.
\end{thm}

By ever-so-slightly extending the main result, we also obtain, under
the same hypothesis regarding $\Sigma$:
\begin{thm}
\label{thm:main result version CPn}Let $L^{n}$ be a monotone, compact,
orientable Lagrangian submanifold of $\left(\Sigma^{2n},\omega\right)$,
such that all the odd-numbered cohomology groups of its universal
cover $\tilde{L}$ vanish.

If $H^{2}\left(\Sigma,\mathbb{Z}\right)$ is generated by $\left[\omega\right]$
or $H^{2}\left(L,\mathbb{Z}\right)=0$, then the Maslov number $N_{L}$
of $L$ is $2$ and there exists some non-trivial $g\in\pi_{1}\left(L\right)$
such that its centraliser is of finite index.
\end{thm}

\begin{cor}
\label{cor:corollaire prinicpal}Let $\left(L_{i}\right)_{i\in I}$
be a finite collection of compact, orientable, $2k+1$-di\-men\-sio\-nal
manifolds such that all the odd-numbered cohomology groups of each
universal cover $\tilde{L}_{i}$ vanish. Assume that either:
\begin{enumerate}
\item $\forall i\in I,H^{2}\left(L_{i},\mathbb{Z}\right)=0$ with $k>1$,
or
\item $H^{2}\left(\Sigma,\mathbb{Z}\right)$ is generated by $\left[\omega_{\Sigma}\right]$.
\end{enumerate}
Then there is no Lagrangian monotone embedding of the connected sum
$\sharp_{i\in I}L_{i}$ in $\Sigma$.
\end{cor}

\begin{cor}
\label{cor:application S1xS2k CP2k+1}Let $p>1$, $k>0$. There is
no monotone embedding of $\left(S^{1}\times S^{2k}\right)^{\sharp p}$
in $\mathbb{C}P^{2k+1}$.
\end{cor}

\paragraph{Outline of the proofs}

Let $L$ be a closed, connected, $K\left(\pi,1\right)$ Lagrangian
in a symplectic manifold $\left(\Sigma,\omega_{\Sigma}\right)$ as
above. In the section~\ref{sec:The-symplectic-model}, borrowing
from P. Biran~\cite{Biran:2001fk}, we will see how we can view most
of $M$ as a complex line bundle over $\Sigma$. In the total space
of this bundle we can associate to $L$ a circle bundle $\Gamma_{L}\rightarrow L$
by considering the points above $L$ of a given modulus. The resulting
$\Gamma_{L}$ is a compact, orientable $K\left(\pi',1\right)$, and
a monotone Lagrangian submanifold of $W=M\setminus\Sigma$. Since
$W$ is assumed to be subcritical, $\Gamma_{L}$ is also displaceable
by an Hamiltonian isotopy.

In the section~\ref{sec:Mihais-work} we will recall some results
obtained in~\cite{Dam-10,Mihai2012} by M. Damian on precisely this
type of Lagrangian. Namely, the Maslov number $N_{\Gamma_{L}}$ of
$\Gamma_{L}$ is $2$ and there exists some non-trivial $g\in\pi_{1}\left(\Gamma_{L}\right)$
such that its centraliser is of finite index. Furthermore, this element
is the boundary of some pseudo-holomorphic disc with Maslov index
$2$.

Then, in section~\ref{sec:Proof}, we obtain a one-to-one correspondence
between the pseudo-holomorphic discs on $\left(\Sigma,L\right)$ and
those on $\left(W,\Gamma_{L}\right)$ with corresponding boundary
and Maslov index $2$. To that end we use the techniques developed
by Biran \& Khanevsky~\cite{BK2011} to project those discs in $M$
down to $\Sigma$ in a holomorphic way, involving some ''stretching
the neck''. This implies the theorem~\ref{prop:Main-result} on
$L$.

This result actually has some interesting consequences, especially
when translated to some looser condition on $L$. These corollaries
are presented in section~\ref{sec:Applications}. Of particular note
is the situation in $\mathbb{C}P^{n}$: as its second cohomology group
in generated by the symplectic form, the Euler class of $\Gamma_{L}$
must vanish. It is thus a trivial $S^{1}$ bundle over $L$. Since
the universal cover of $\Gamma_{L}$ retracts to the one of $L$,
those two covers have the same cohomology. In particular, if odd-numbered
cohomology groups of the universal cover $\tilde{L}$ vanish the same
is true for $\Gamma_{L}$. This condition on $L$ is sufficient to
apply the results from \cite{Mihai2012} and therefore, in $\mathbb{C}P^{n}$,
sufficient to obtain the same conclusion as in~\ref{prop:Main-result}.
In some examples, such as $\left(S^{1}\times S^{2k}\right)^{\sharp p}$
in $\mathbb{C}P^{2k+1}$, it is incompatible with the structure of
the fundamental group, making it impossible to embed as a monotone
Lagrangian.

\paragraph{Acknowledgements}

I want to thank Mihai Damian for his incredible patience as he guided
me through the writing of this article. I also wish to thank Michael
Khanevsky and Romain Ponchon for the time they took to discuss various
parts of it. Thanks to Florian Delage for his thoughtful proofing.

\section{\label{sec:The-symplectic-model}The symplectic model}

\subsection{\label{sub:Standard-symplectic-bundle}Standard symplectic bundle}

Let $\left(\Sigma,\tau\right)$ be a closed integral symplectic manifold,
i.e. such that $\left[\tau\right]\in H^{2}\left(\Sigma,\mathbb{Z}\right)$
is well defined. We now present the standard symplectic bundle, introduced
by Paul Biran in~\cite{Biran:2001fk}. This will be our model to
understand most of $M$ as a symplectic manifold.

Let $\mathcal{N}$ be a complex line bundle over $\Sigma$ with $\left[\tau\right]$
as its first Chern class. On $\mathcal{N}$ choose an Hermitian metric
$\left|.\right|$, an Hermitian connection $\nabla$ and denote by
$H^{\nabla}$ the associated horizontal subbundle. The transgression
1-form $\alpha$ is then defined out of the zero section by:
\begin{eqnarray*}
\alpha_{\big|H^{\nabla}}=0 & \alpha_{\left(p,u\right)}\left(u\right)=0 & \alpha_{\left(p,u\right)}(iu)=\frac{1}{2\pi}
\end{eqnarray*}
where $\left(p,u\right)\in\mathcal{N}$. Then $\text{d}\alpha=-\pi^{*}\tau$
with $\mathcal{N}\overset{\pi}{\longrightarrow}\Sigma$. Designating
by $r$ the distance to the zero section induced by the metric, the
standard symplectic form on $\mathcal{N}\setminus0_{\Sigma}$ is then:
\[
\omega_{std}=-\text{d}\left(e^{-r^{2}}\alpha\right)=e^{-r^{2}}\pi^{*}\tau+2re^{-r^{2}}\text{d}r\wedge\alpha
\]
Remark that in the vertical direction, $r\text{d}r\wedge\text{d}\alpha$
is the usual symplectic form on $\mathbb{C}$. Hence each summand
can be extended to the whole total space $\mathcal{N}$ and is symplectic
on the horizontal and vertical subbundle respectively.

The standard symplectic disc bundle is given by:
\[
E_{r}=\left\{ \left(p,u\right)\in\mathcal{N},\left|u\right|\leqslant r\right\} 
\]
endowed with the restriction of $\omega_{std}$.
\begin{defn}
An isotropic CW-complex in $M$ is some subset that is homeomorphic
to a CW-complex in such a way that the interior of each cell is isotropically
embedded in $M$.\end{defn}
\begin{thm}
\label{(Biran)Model}(Theorem 1.A in \cite{Biran:2001fk}) Let $\left(M^{2n+2},\omega\right)$
be a closed integral Kähler manifold and let $\Sigma\subset M$ be
a complex hypersurface whose homology class $\left[\Sigma\right]\in H_{2n}\left(M;\mathbb{Z}\right)$
is Poincaré-dual to a multiple $k\left[\omega\right]$ of $\left[\omega\right]\in H^{2}\left(M,\mathbb{Z}\right)$.
Then, there exists an isotropic CW-complex $\Delta\subset M$ whose
complement \textemdash{}~the open dense subset $\left(M\setminus\Delta,\omega\right)$~\textemdash{}
is symplectomorphic to the standard symplectic bundle $\left(\mathcal{N},\frac{1}{k}\omega_{std}\right)$
over $\Sigma$ pertaining to $\tau=k\omega_{\big|\Sigma}$.

In other words, there exists an embedding $F:\left(\mathcal{N},\frac{1}{k}\omega_{std}\right)\hookrightarrow\left(M,\omega\right)$
such that
\begin{itemize}
\item The zero section is isomorphic to $\Sigma$.
\item $\Delta=M\setminus F(\mathcal{N})$ is an isotropic CW-complex.
\item $\forall r>0$, $\left(M\setminus F\left(\overset{\circ}{E_{r}}\right),\omega\right)$
is a Weinstein domain.
\end{itemize}
\end{thm}
\begin{rem}
As an immediate consequence of it being isotropic the dimension of
$\Delta$ is at most half of $M$'s. This simple fact will be useful
it in the proof of proposition~\ref{etirement cou}.
\end{rem}
Subsequently we will perform a small abuse of notation and denote
by $\pi:M\setminus\Delta\to\Sigma$ the composition $\pi\circ F^{-1}$.
\begin{example}
$\mathbb{C}P^{n}$, seen as the hyperplane $\left\{ z_{0}=0\right\} $
in $\mathbb{C}P^{n+1}$ constitutes an example with $k=1$.\end{example}
\begin{rem}
More recently, in~\cite{BK2011}, the notion of symplectic hyperplane
section was introduced. The results of this paper should hold in that
framework.
\end{rem}

\subsection{\label{sub:Lagrangian-circle-bundle}The Lagrangian circle fibration}

Within the previous framework, we denote by $P_{r_{0}}$ the $S^{1}$-bundle
over $\Sigma$ given by elements of radius $r_{0}$, with the projection
$\pi_{r_{0}}:P_{r_{0}}\to\Sigma$. Then we define $\Gamma_{L}$ as
$\pi_{r_{0}}^{-1}(L)$. Its tangent bundle can be locally decomposed
as the tangent of $L$ in $H^{\nabla}$ and of a circle in $\mathbb{C}$
vertically, hence it is a Lagrangian submanifold of $M$ and of $W=M\setminus\Sigma$.

We will now consider the morphism induced by $\pi:\left(W\setminus\Delta,\Gamma_{L}\right)\to\left(\Sigma,L\right)$
on the second relative homotopy groups. For a Lagrangian submanifold
$\Lambda$ of a symplectic manifold $\left(V,\omega\right)$ we denote
by $\mu_{K}:\pi_{2}\left(V,K\right)\to\mathbb{Z}$ the Maslov index.
Denote by $\iota:W\setminus\Delta\to W$ the inclusion.
\begin{prop}
(Proposition 4.1.A in \cite{Biran-InterPub}) \label{passage de maslov}If
$\dim_{\mathbb{C}}\Sigma>1$ or $W$ is subcritical:
\begin{enumerate}
\item The morphism $\iota_{*}:\pi_{2}\left(W\setminus\Delta\right)\to\pi_{2}\left(W\right)$
induced by the inclusion is surjective. When $\dim_{\mathbb{C}}\Sigma>2$,
it is an isomorphism.
\item for every $B\in\pi_{2}\left(W\setminus\Delta,\Gamma_{L}\right)$,
\[
\mu_{\Gamma_{L}}\left(B\right)=\mu_{L}\left(\pi_{*}B\right)
\]
In particular, if $L\subset\Sigma$ is monotone then $\Gamma_{L}\subset W$
is monotone too, with the same minimal Maslov number.
\end{enumerate}
\end{prop}

\section{Results from the lifted Floer homology\label{sec:Mihais-work}}

Recall that in our setup, $L\subset\Sigma$ is monotone, hence $\Gamma_{L}\subset W$
too. Since $L$ is a compact $K\left(\pi,1\right)$ and $\Gamma_{L}$
a circle bundle over $L$ it is also a compact $K\left(\pi',1\right)$.
In particular, odd-numbered cohomology groups of its universal cover
$\tilde{\Gamma}_{L}$ vanish. Given a volume form $\upsilon$ on $L$,
we can construct a volume on $\Gamma_{L}$ by pulling back $\upsilon$
and wedging the result with $\alpha^{\nabla}$. Therefore, $\Gamma_{L}$
is orientable. Besides, we assumed $W$ to be subcritical: in particular
every compact subset is Hamiltonian displaceable (see \cite{Biran-InterPub}),
and this applies to $\Gamma_{L}$.
\begin{thm}
\label{Th=0000E9or=0000E8me Damian gamma}\cite{Dam-10} Let $\Lambda$
be a monotone, compact, orientable, Hamiltonian displaceable, Lagrangian
submanifold. Assume further that the odd-num\-bered cohomology groups
of its universal cover $\tilde{\Lambda}$ vanish. Then:
\begin{enumerate}
\item Its minimal Maslov number is $N_{\Lambda}=2$.
\item For any generic almost complex structure $J$ there exist $p\in\Lambda$
and a non-trivial $g\in\pi_{1}\left(\Lambda,p\right)$ such that the
number of pseudo-holomorphic discs $u$ evaluating in $p$ and verifying
$\left[\partial u\right]=g$ and $\mu_{\Lambda}\left(u\right)=2$
is (finite and) odd.
\end{enumerate}
\end{thm}

\begin{prop}
\label{prop:Mihai Gamma reciproque}\cite{Mihai2012} When some non-trivial
$g\in\pi_{1}\left(\Lambda\right)$ complies with the result numbered
$2$ in~\ref{Th=0000E9or=0000E8me Damian gamma} as an hypothesis,
then its centraliser is of finite order.
\end{prop}
Now the purpose of the next part will be to repatriate this result
down on $L$ by constructing a one-to-one correspondence between the
pseudo-holomorphic discs on $\left(\Sigma,L\right)$ and those on
$\left(W,\Gamma_{L}\right)$ with corresponding boundary and Maslov
index $2$.

\section{Proof of the main theorem\label{sec:Proof}}

The crucial element in projecting the discs obtained by Damian's theorem~\ref{Th=0000E9or=0000E8me Damian gamma}
is the obtention of a suitable almost complex structure on $W$ through
the procedure of ''stretching the neck''. This procedure was first
presented in~\cite{BEH+}, and refined for our setting in~\cite{BK2011}.
The idea is that by modifying the almost complex structure on $W$
we can prevent pseudo-holomorphic discs of bounded energy -- such
as our Maslov-2 index ones in this monotone context -- to thread their
way too far from $P_{r_{0}}$ and in particular to approach $\Delta$.
As a consequence, their projection will be well-defined.

\subsection{\label{sub:Stretching-the-neck}Stretching the neck}

Let us begin by taking a generic almost complex structure $J_{\Sigma}$
on $\Sigma$, which is tamed by $\omega_{\Sigma}$. We will here again
use the notation 
\[
E_{r}=\left\{ \left(p,u\right)\in\mathcal{N},\left|u\right|\leqslant r\right\} 
\]
for the closed disc bundle of radius $r$ in $\mathcal{N}$.

Let us choose some $\epsilon>0$ such that the restriction of $F$
(as defined by Biran's theorem~\ref{(Biran)Model}) to $E_{r_{0}+\epsilon}$
is a diffeomorphism, where $r_{0}$ is the radius used to define $\Gamma_{L}$
in subsection~\ref{sub:Lagrangian-circle-bundle}. Then the complement
$U$ of $E_{r_{0}+\epsilon}$ in $M$ is a neighbourhood of $\Delta$
-- and the part of $M$ we want the discs to avoid.

We set an almost complex structure $J_{\mathcal{N}}$ on $\mathcal{N}$,
defined along the horizontal subbundle $H^{\nabla}$ as the ``pull-back''
of $J_{\Sigma}$ by $\pi$
\[
\forall v\in H^{\nabla},J_{\mathcal{N}}\left(v\right)=\left(T\pi_{\big|H^{\nabla}}\right)^{-1}\circ J_{\Sigma}\circ T\pi\left(v\right)
\]
and along the fibres as the multiplication by $i$. We then push it
by $F$ on $E_{r_{0}+\epsilon}=M\setminus U$ and denote by $J_{M}$
a generic extension on $M$ taming $\omega$. $J_{W}$ will denote
its restriction on $W$. 

Recall that $P$ is the bundle over $\Sigma$ of $r_{0}$-radius circles
in $\mathcal{N}$ which -- since it lies within $E_{r_{0}+\epsilon}$
-- can be thought as being in either $M$ or $W$. We will thereafter
consider the two connected components of respectively $M\setminus P$
and $W\setminus P$ with the following sign convention: $\Sigma\subset M^{+}$,
$\Delta\subset U\subset W^{-}=M^{-}$. For $R>0$ we put:
\[
W^{R}=W^{-}\underset{\left\{ -R\right\} \times P}{\bigcup}\left[-R,R\right]\times P\underset{\left\{ R\right\} \times P}{\bigcup}W^{+}
\]

On $W^{R}$ the almost complex structure is defined as $J_{W}$ on
$W^{-}$ and $W^{+}$, and by translation invariance in the middle
part. The resulting structure is only continuous on the glued boundaries
but can be slightly deformed near them to a smooth almost complex
structure which we denote by $J^{R}$. Furthermore, this smoothing
can be achieved using only the radial coordinate $t\in\left[-R,R\right]$
and the angular one $\theta$ in the (circle) fibre of $P$, so as
to be invisible once projected to $\Sigma$.

To push back $J^{R}$ on $W$, we make use of a (decreasing) diffeomorphism
$\varphi_{R}:\left[-R-\epsilon,R\right]\to\left[r_{0},r_{0}+\epsilon\right]$
such that its derivative satisfies $\varphi_{R}'\left(t\right)=-1$
near the boundary of $\left[-R-\epsilon,R\right]$. We set the diffeomorphism:
\[
\lambda_{R}:W^{R}\to W
\]
to be the identity on both $W^{+}$ and $U$, and between $\left[-R-\epsilon,R\right]\times P$
and $\left[r_{0},r_{0}+\epsilon\right]\times P$ to be induced by
$\varphi_{R}$ on the first coordinate. Here we made use of the identification
$W^{-}\setminus U\approx]r_{0},r_{0}+\epsilon]\times P$. Note that
$\lambda_{R}$ preserves the projection and the angular coordinate
wherever defined. Finally, we define $J_{R}$ on $W$ as the push-forward
$\left(\lambda_{R}\right)_{*}J^{R}$; it happens to tame $\omega$.
Besides $\pi$ is -- by construction of $J_{R}$ -- $J_{R}\text{-}J_{\Sigma}\text{-}$holomorphic
outside $U$.
\begin{prop}
\label{etirement cou}Suppose the minimal Chern number of $\Sigma$
is $N_{\Sigma}\geqslant2$. Let $p$ be a point in $\Gamma_{L}$.

Then there exists $R_{0}>0$ such that for every $J_{R}$ as described
above with $R>R_{0}$, every Maslov-$2$ $J_{R}$-holomorphic disc
$u:\left(D,\partial D\right)\to\left(W,\Gamma_{L}\right)$ passing
through $p$ is contained in the image $F\left(E_{r_{0}+\epsilon}\right)$.\end{prop}
\begin{proof}
Below we will follow the reasoning of~\cite{BK2011} and refer to
the results of~\cite{BEH+}, which also hold for holomorphic curves
with boundary on Lagrangian submanifolds.

We will assume by contradiction that for a generic almost complex
structure $J_{\Sigma}$ on $\Sigma$, there exists a sequence $R_{n}>0$
going to infinity with for every $n\in\mathbb{N}$ a Maslov-$2$ $J_{R_{n}}$-holomorphic
disc $u_{n}'$ in $W$ with its boundary on $\Gamma_{L}$ that leaves
the image of the $\left(r_{0}+\epsilon\right)$-disk bundle. We will
denote $J_{R_{n}}$ by $J_{n}$.

In~\cite{BEH+} Bourgeois, Eliashberg, Hofer et al. give a sense
to the idea of convergence for our sequence of pseudo-holomorphic
discs. In addition, they establish the compactness of a moduli space
where $\left(u'_{n}\right)$ lives, provided that its ``total energy''
is uniformly bounded.

Our sequence of discs leaving $E_{r_{0}+\epsilon}$ are the $u'_{n}:\left(D^{2},\partial D\right)\rightarrow\left(W,\Gamma_{L}\right)$,
and we denote by $u_{n}=\lambda_{R_{n}}^{-1}\circ u'_{n}:\left(D,\partial D\right)\to\left(W^{R_{n}},\Gamma_{L}\right)$
the same discs seen in $W^{R_{n}}$. We wish to establish a uniform
bound to the total energy of $\left(u_{n}\right)$. First, the $\omega$-energy
of some $J_{R}$-holomorphic $\mbox{u : \ensuremath{\left(D,\partial D\right)\to\left(W^{R},\Gamma_{L}\right)}}$
is:
\[
E_{\omega}\left(u\right)=\int_{u^{-1}\left(W^{+}\cup W^{-}\right)}u^{*}\omega+\int_{u^{-1}\left(\left[-R,R\right]\times P\right)}u^{*}p_{P}^{*}\omega
\]
where $p_{P}$ is the projection $\left[-R,R\right]\times P\to P$.
We wish to compare this quantity to $\int_{u^{-1}\left(\left[-R,R\right]\times P\right)}u'^{*}\omega$.
Since $\lambda_{R}$ on $\left[-R,R\right]\times P$ maps to $E_{r_{0},r_{0}+\epsilon}$
on which $\omega$ is canonical, we have: 
\[
\left(\lambda_{R}{}^{*}\omega\right)_{\big|\left[-R,R\right]\times P}=e^{-\varphi_{R}\left(t\right)^{2}}\pi_{\Sigma}^{*}\omega_{\Sigma}+2re^{-r^{2}}\text{d}r\wedge\alpha^{\nabla}
\]
Let us split $\int_{u^{-1}\left(\left[-R,R\right]\times P\right)}u'^{*}\omega$
as a sum and consider the first addend:
\[
\int_{u^{-1}\left(\left[-R,R\right]\times P\right)}u'^{*}\left(2re^{-r^{2}}\text{d}r\wedge\alpha^{\nabla}\right)=2\int_{u^{-1}\left(\left[-R,R\right]\times P\right)}u'^{*}\left(e^{-r^{2}}\right)u'^{*}\left(r\text{d}r\wedge\alpha^{\nabla}\right)
\]
is non-negative since $u'$ is $J_{R}$-holomorphic and $J_{R}$ is,
on the fibre, the usual product by $i$: the part $u'^{*}\left(r\text{d}r\wedge\alpha^{\nabla}\right)$
can be seen as a square norm.

Meanwhile, $e^{-\varphi_{R}\left(t\right)^{2}}\geqslant e^{-\left(r_{0}+\epsilon\right)}$
implies
\[
\int_{u^{-1}\left(\left[-R,R\right]\times P\right)}u^{*}\left(e^{-\varphi_{R}\left(t\right)^{2}}\pi_{\Sigma}^{*}\omega_{\Sigma}\right)\geqslant\int_{u^{-1}\left(\left[-R,R\right]\times P\right)}u^{*}\left(e^{-\left(r_{0}+\epsilon\right)}\pi_{\Sigma}^{*}\omega_{\Sigma}\right)
\]
since $\pi_{\Sigma}\circ u$ is $J_{\Sigma}$-holomorphic when defined,
given our choice of $J_{R}$. Combining the two inequalities, we have
\[
\int_{u^{-1}\left(\left[-R,R\right]\times P\right)}u'^{*}\omega\geqslant\int_{u^{-1}\left(\left[-R,R\right]\times P\right)}u^{*}p_{P}^{*}\omega
\]
And finally:
\[
\int_{D^{2}}u'^{*}\omega\geqslant E_{\omega}\left(u\right)
\]

Let's now remember that $\Gamma_{L}$ is monotone in $W$ and $u'_{n}$
has a Maslov index of $2$, hence the integral of $\omega$'s pull-back
by $\left(u'_{n}\right)$ is constant, which yields a bound on the
$\omega$-energies of $\left(u_{n}\right)$.

Now by lemma $9.2$ of~\cite{BEH+}, this bound implies one on the
total energy of $\left(u_{n}\right)$. Hence its main theorem $10.6$
holds and a subsequence of $\left(u_{n}\right)$ converges to a so-called
stable holomorphic building $\bar{u}$. Abusing the notation, we will
now refer to a converging subsequence by $\left(u_{n}\right)$.

To describe this limit, put $W_{\infty}^{+}=]-\infty,0]\times P\bigcup W^{+}$
and then $W_{\infty}^{-}=[0,+\infty[\times P\bigcup W^{-}$ respectively
glued on their boundaries. Extend $J_{W}$ on the cylindrical parts
$]-\infty,0]\times P$ $[0,+\infty[\times P$ by invariance under
translation and smooth it as $J_{R}$ was smoothed near the glued
parts. The disjoint union endowed with this almost-complex structure
will be denoted by $\left(W^{\infty},J_{\infty}\right)$, and can
be considered as the limit of $\left(W^{R},J_{R}\right)$ as $R\to\infty$.
On $\mathbb{R}\times P$ cylinders, $J_{\infty}$ is likewise defined
by invariance under translation.

Now $\bar{u}$ is a disconnected $J_{\infty}$-holomorphic curve which
consists of the following connected components:
\begin{itemize}
\item A base $J_{\infty}$-holomorphic map $u_{+}:\left(S_{+},\partial S_{+}\right)\to\left(W_{\infty}^{+},\Gamma_{L}\right)$,
where $S_{+}$ is a disc with one or more punctures. Near these punctures
$u_{+}$ is asymptotically cylindrical and converges to a periodic
orbit of the Reeb vector field of $\left(P,\alpha\right)$, where
$\alpha$ is the transgression $1$-form made explicit in~\ref{sub:Standard-symplectic-bundle}.
Given the choice of $\alpha$ the periodic orbits of the Reeb vector
field are precisely the fibres of the circle bundle $P\to\Sigma$.
\item A number of intermediate $J_{\infty}$-holomorphic maps $u_{i}:S_{i}\to\mathbb{R}\times P$
where each $S_{i}$ is a sphere with one or more punctures. Near those
punctures, the $u_{i}$ are asymptotically cylindrical with Reeb orbits
sections as well.
\item Some capping $J_{\infty}$-holomorphic maps, each of the form $u_{-}:S_{-}\to W_{\infty}^{-}$
where $S_{-}$ is a sphere with one or more punctures. $u_{-}$ is
asymptotically cylindrical near each puncture in a similar way to
$u_{+}$. To simplify the notation we will assume that there exists
one such map; in the case there are many, the argument is the same.
\end{itemize}
Moreover, those components fit over the punctures, i.e. to each asymptotical
cylinder corresponds another, with the same base orbit, in the other
direction on the $\mathbb{R}$ component. As such they can be glued,
and the result remains a topological disc. We wish to compute, component
by component, the Maslov index of $\bar{u}$.

By the definition of $J_{\infty}$ on $W_{\infty}^{+}$, the projection
$\pi_{+}:W_{\infty}^{+}\to\Sigma$ is $\left(J_{\infty},J_{\Sigma}\right)$-\-ho\-lo\-mor\-phic,
hence $\pi_{+}$ sends $u_{+}$ to a punctured disc $\pi_{+}\circ u_{+}:\left(S_{+},\partial S_{+}\right)\to\left(\Sigma,L\right)$.
The periodic orbits at infinity are projected by $\pi_{+}$ to single
points in $\Sigma$ since they are the fibres of the circle bundle
$P\to\Sigma$. Since the convergence near the puncture holds in the
$C^{1}$ norm and the limit has bounded energy, they give rise to
removable singularities on $\pi_{+}\circ u_{+}$. Therefore $\pi_{+}\circ u_{+}$
becomes a genuine $J_{\Sigma}$-holomorphic disc.

Likewise, the intermediate maps can be projected to $\Sigma$ by forgetting
the $\mathbb{R}$ coordinate and projecting $P$. By the same argument,
the singularities are removable and we obtain $J_{\Sigma}$-holomorphic
spheres.

Alas, we cannot straightly use this method for $u_{-}$, for it may
intersect the isotropic skeleton $\Delta$ and the projection, even
where defined, has no reason to be holomorphic. Since $\text{codim\,}\Delta>2=\dim u_{-}$,
we can at least perturb homotopically its part in $W^{-}\subset W_{\infty}^{-}$
so that the resulting $\tilde{u}_{-}$, while no more holomorphic,
avoids $\Delta$. We can now project this perturbed curve to $\Sigma$;
as before the singularities are removable and we obtain a sphere $v:S^{2}\to\Sigma$.
We claim that $v$ has a positive Chern number; since $\Sigma$ is
monotone it suffices to show that its symplectic area is positive.
We have:
\[
\int\limits _{S^{2}}v^{*}\omega_{\Sigma}=\int\limits _{S_{-}}\tilde{u}_{-}^{*}\pi_{-}^{*}\omega_{\Sigma}=\int\limits _{\tilde{u}_{-}^{-1}\left(W^{-}\right)}\tilde{u}_{-}^{*}\pi_{-}^{*}\omega_{\Sigma}+\int\limits _{\tilde{u}_{-}^{-1}\left(W_{\infty^{-}}\setminus W^{-}\right)}\tilde{u}_{-}^{*}\pi_{-}^{*}\omega_{\Sigma}
\]
Since $u_{-}$ is not perturbed on $W_{\infty^{-}}\setminus W^{-}$,
where $\pi_{-}$ is besides $\left(J_{\infty},J_{\Sigma}\right)$-ho\-lo\-mor\-phic,
the second addend is positive. For the first addend:
\begin{eqnarray*}
\int\limits _{\tilde{u}_{-}^{-1}\left(W^{-}\right)}\tilde{u}_{-}^{*}\pi_{-}^{*}\omega_{\Sigma} & = & \int\limits _{\tilde{u}_{-}^{-1}\left(W^{-}\setminus\Delta\right)}\tilde{u}_{-}^{*}\left(-\text{d}\alpha^{\nabla}\right)\\
 & = & e^{\left(r_{0}+\epsilon\right)^{2}}\int\limits _{\partial\tilde{u}_{-}^{-1}\left(W^{-}\setminus\Delta\right)}\tilde{u}_{-}^{*}\left(-e^{-\left(r_{0}+\epsilon\right)^{2}}\alpha^{\nabla}\right)\\
 & = & e^{\left(r_{0}+\epsilon\right)^{2}}\int\limits _{\tilde{u}_{-}^{-1}\left(W^{-}\setminus\Delta\right)}\tilde{u}_{-}^{*}\text{d}\left(-e^{-r^{2}}\alpha^{\nabla}\right)\\
 & = & e^{\left(r_{0}+\epsilon\right)^{2}}\int\limits _{\tilde{u}_{-}^{-1}\left(W^{-}\right)}\tilde{u}_{-}^{*}\omega\\
 & = & e^{\left(r_{0}+\epsilon\right)^{2}}\int\limits _{u_{-}^{-1}\left(W^{-}\right)}u_{-}^{*}\omega
\end{eqnarray*}
which is positive since $u_{-}$ is $J_{\infty}$-holomorphic. Finally,
$c_{1}^{\Sigma}\left(\left[v\right]\right)>0$.

It suffices to use that:
\[
2=\mu_{\Gamma_{L}}\left(\bar{u}\right)=\mu_{L}\left(\left[\pi_{+}\circ u_{+}\right]\right)+\sum\limits _{i}2c_{1}^{\Sigma}\left(\left[\pi\circ u_{i}\right]\right)+2c_{1}^{\Sigma}\left(\left[v\right]\right)
\]
with all the Chern classes positive. Given that $\pi_{+}\circ u_{+}$
is $J_{\Sigma}$-holomorphic, its Maslov index is non-negative; and
$N_{\Sigma}\geqslant2$ implies that each Chern class is at least
$2$. We have reached a contradiction.
\end{proof}

\subsection{Regularity of the almost complex structure}

The result we use from~\cite{Dam-10} is derived from a flavour of
Floer homology called the \emph{lifted Floer homology}. As its older,
non-lifted counterpart, it requires regularity of the almost complex
structure in the sense given by McDuff and Salamon in~\cite{mcduff2004j}.

Recall that the choice of almost complex structure $J_{R}$ on $W^{R}$
in~\ref{sub:Stretching-the-neck} was only partially free. More specifically
we could choose any structure taming $\omega$ on a neighbourhood
$U$ of $W^{-}$, while on $E_{r_{0}+\epsilon}$ the definition was
split between pulling back on the horizontal distribution $H^{\nabla}$
some almost complex structure $J_{\Sigma}$ taming $\omega_{\Sigma}$,
and multiplying by $i$ in the fibres.

We aim to establish the surjectivity of the linearization of the $\bar{\partial}$\textendash{}operator
$D_{u}$ at each $J_{R}$-holomorphic disk $u:\left(D,\partial D\right)\to\left(W^{R},\Gamma_{L}\right)$.
Since the almost complex structure on $U$ can be chosen arbitrarily,
the general theory establishes regularity for discs going out of $E_{r_{0}+\epsilon}$.

To treat the discs staying within $E_{r_{0}+\epsilon}$, we first
remark that $\lambda_{R}$ identifies $\left(E_{r_{0}+\epsilon},J_{R}\right)$
with $\left(E'_{R},J^{R}\right)$ where $E'_{R}=\left[-R,R\right]\times P\cup_{\left\{ R\right\} \times P}W^{+}$.
We will reason within this later setting. Let us denote by $J_{H}$
the restriction of $J^{R}$ on $H^{\nabla}$, which is essentially
$J_{\Sigma}$ through the identification given by $T\pi$. The definition
of $J^{R}$ can be written as: $\left(TE'_{R},J^{R}\right)\simeq\left(H^{\nabla},J_{H}\right)\oplus\pi^{*}\left(\mathcal{N},i\right)$.
Let $u$ be a disc that stays inside $E'_{R}$, with $j$ being the
complex structure on the unit disc. We have:
\[
\bar{\partial}_{J^{R}}\left(u\right)=\frac{1}{2}\left(Tu^{\big|\pi^{*}\mathcal{N}}+\left(\pi^{*}i\right)\circ Tu^{\big|\pi^{*}\mathcal{N}}\circ j\right)\oplus\frac{1}{2}\left(Tu^{\big|H^{\nabla}}+J_{H}\circ Tu^{\big|H^{\nabla}}\circ j\right)
\]
We notice that
\[
Tu^{\big|H^{\nabla}}+J_{H}\circ Tu^{\big|H^{\nabla}}\circ j=\left(T\pi_{\big|H^{\nabla}}\right)^{-1}\left(T\left(\pi\circ u\right)+J_{\Sigma}\circ T\left(\pi\circ u\right)\circ j\right)
\]
is zero if and only if $\pi\circ u$ is $J_{\Sigma}$-holomorphic.

Just as the operator $\bar{\partial}_{J_{R}}$, the vector bundle
$\mathcal{E}$ of smooth $J_{R}$-anti-linear $1$-forms over the
smooth maps $\left(D^{2},S^{1}\right)\to\left(W^{R},\Gamma_{L}\right)$
splits as a direct sum:
\[
\mathcal{E}_{u}\approx\Omega_{i}^{0,1}\left(u^{*}\pi^{*}\mathcal{N}\right)\oplus\Omega_{J_{H}}^{0,1}\left(u^{*}H^{\nabla}\right)
\]
The right addend effectively corresponds to the smooth $J_{\Sigma}$-anti-linear
$1$-forms over the smooth maps $\left(D^{2},S^{1}\right)\to\left(\Sigma,L\right)$.
To the right hand of the $\bar{\partial}_{J_{R}}$ we can associate
a linearization which is surjective if and only if the linearization
$\bar{\partial}_{J_{\Sigma}}$ is surjective. This last point is achieved
thanks to the genericity of $J_{\Sigma}$.

To deal with the left addend, we can consider a holomorphic trivialisation
$g:\left(\pi\circ u\right)^{*}\mathcal{N}\to D\times\mathbb{C}$.
This yields the identifications: 
\[
\Omega_{i}^{0,1}\left(\left(\pi\circ u\right)^{*}\mathcal{N}\right)\approx_{g}\Omega^{0,1}\left(\mathbb{C}\right)
\]
and 
\[
\frac{1}{2}\left(Tu^{\big|\pi^{*}\mathcal{N}}+\left(\pi^{*}i\right)\circ Tu^{\big|\pi^{*}\mathcal{N}}\circ j\right)=g^{-1}\circ\bar{\partial}
\]
Since this almost complex structure (multiplication by $i$) is regular
we obtain the surjectivity on the left-hand part.

\subsection{Uniqueness of the pseudo-holomorphic discs lifting}

Let us first recall the lemma $7.1.1$ of \cite{BK2011}:
\begin{prop}
\label{relevement disques}Let $u:\left(D^{2},S^{1}\right)\to\left(\Sigma,L\right)$
be a $J_{\Sigma}$-holomorphic disc. Given $\xi\in S^{1}$ and $\tilde{p}\in\Gamma_{L}\cap\pi^{-1}\left(u(\xi)\right)$
there is a unique $J_{\mathcal{N}}$-holomorphic lift $\tilde{u}:\left(D^{2},S^{1}\right)\to\left(\mathcal{N}\setminus\Sigma,\Gamma_{L}\right)$
of $u$ such that $\tilde{u}(\xi)=\tilde{p}$.
\end{prop}
We can actually check that the boundary of those lifted discs is as
required, as we have more precisely:
\begin{prop}
Let $u:\left(D^{2},S^{1}\right)\to\left(\Sigma,L\right)$ be a $J_{\Sigma}$-holomorphic
disc such that $\mu_{L}\left(u\right)=2$, $\left[\partial u\right]=\pi_{*}\tilde{g}$
and $u$ passes through $p\in L$. If $\tilde{u}$ is the pseudo-holomorphic
lift of $u$ passing through $\tilde{p}\in\Gamma_{L}\cap\pi^{-1}\left(p\right)$,
then $\mu_{\Gamma_{L}}\left(\tilde{u}\right)=2$ and $\left[\partial\tilde{u}\right]=\tilde{g}$.\end{prop}
\begin{proof}
We know there is an odd natural number, which is in particular positive,
of $J_{\mathcal{N}}$-holomorphic curves $\left(D^{2},S^{1}\right)\to\left(W,\Gamma_{L}\right)$
passing through $\tilde{p}$, with Maslov index $2$ and boundary
in $\tilde{g}$. Let us denote one by $\tilde{u}_{0}$.

Using proposition~\ref{passage de maslov}, we have that $\mu_{\Gamma_{L}}\left(\tilde{u}\right)=\mu_{L}\left(u\right)=2$.
By our main argument of neck-stretching~\ref{etirement cou} both
$\tilde{u}_{0}$ and $\tilde{u}$ avoid $\Delta$.

Let us define 
\begin{eqnarray*}
\gamma:S^{1} & \longrightarrow & \Gamma_{L}\\
v & \longmapsto & v.\tilde{p}
\end{eqnarray*}
so that $\ker\pi_{*}=\langle\left[\gamma\right]\rangle\subset\pi_{1}\left(\Gamma_{L}\right)$.
Since $\pi_{*}\left[\partial\tilde{u}\right]=\left[\partial u\right]=g=\pi_{*}\tilde{g}$,
there is some $l\in\mathbb{Z}$ such that $\left[\partial\tilde{u}\right]=\tilde{g}\left[\gamma\right]^{l}$.
Furthermore, since $\mu_{\Gamma_{L}}\left(\tilde{u}\right)=\mu_{\Gamma_{L}}\left(\tilde{u}_{0}\right)$
by monotonicity we have:
\[
\int\limits _{D^{2}}\tilde{u}_{0}^{*}\omega=\int\limits _{D^{2}}\tilde{u}^{*}\omega
\]
Since we avoid $\Delta$, $\omega=-\text{d}\left(e^{-r^{2}}\alpha^{\nabla}\right)$
and by applying the Stokes formula:
\[
e^{-r_{0}^{2}}\int\limits _{S^{1}}\partial\tilde{u}_{0}^{*}\alpha^{\nabla}=e^{-r_{0}^{2}}\int\limits _{S^{1}}\partial\tilde{u}^{*}\alpha^{\nabla}
\]
Hence 
\[
\int\limits _{S^{1}}\partial\tilde{u}_{0}^{*}\alpha^{\nabla}=\int\limits _{S^{1}}\partial\tilde{u}_{0}^{*}\alpha^{\nabla}+l\int\limits _{S^{1}}\gamma^{*}\alpha^{\nabla}
\]
Since $\int\limits _{S^{1}}\gamma^{*}\alpha^{\nabla}>0$, $l=0$ so
$\left[\partial\tilde{u}\right]=\tilde{g}$.
\end{proof}

\section{\label{sec:Applications}Applications}

This section will be dedicated to the proof of theorem~\ref{thm:main result version CPn},
and corollaries~\ref{cor:corollaire prinicpal} and~\ref{cor:application S1xS2k CP2k+1},
as stated in the introduction.

Recall that the minimal hypothesis for theorem \ref{Th=0000E9or=0000E8me Damian gamma}
and proposition \ref{prop:Mihai Gamma reciproque} concerning the
topology of the Lagrangian is actually less stringent than being a
$K\left(\pi,1\right)$: it is enough for all of the odd-numbered cohomology
groups of its universal cover to vanish. As the rest of our proof
is not affected by this change, this is the condition we will look
for henceforth.

In particular $L$ is not assumed to be a $K\left(\pi,1\right)$ anymore
unless specified.

\subsection{On the triviality of $\Gamma_{L}$}

Let us begin with a direct application of this weaker condition:
\begin{prop}
If $\Gamma_{L}$ is a trivial circle bundle over $L$, and all the
odd-numbered cohomology groups of its universal cover $\tilde{L}$
vanish, then the Maslov number $N_{L}$ of $L$ is $2$ and there
exists some non-trivial $g\in\pi_{1}\left(L\right)$ such that its
centraliser is of finite index.\end{prop}
\begin{proof}
If $\Gamma_{L}$ is a trivial bundle, then $\tilde{L}$ is a retraction
of $\tilde{\Gamma}_{L}=\tilde{L}\times\mathbb{R}$, and the cohomology
of $\tilde{\Gamma}_{L}$ is exactly the same as $\tilde{L}$. In particular,
the odd-numbered cohomology groups of $\tilde{\Gamma}_{L}$ vanish.
The hypothesis of theorem \ref{Th=0000E9or=0000E8me Damian gamma}
are now valid on $\Gamma_{L}$, hence the result is obtained there.
The one-to-one correspondence built in section~\ref{sec:Proof} between
the pseudo-holomorphic discs on $\left(\Sigma,L\right)$ and those
on $\left(W,\Gamma_{L}\right)$ with corresponding boundary and Maslov
index $2$ still exists, as we did not use any assumption on our Lagrangian
topology in its proof. Therefore, we can apply proposition~\ref{prop:Mihai Gamma reciproque}
to $L$.\end{proof}
\begin{lem}
If $H^{2}\left(L,\mathbb{Z}\right)=0$ or $H^{2}\left(\Sigma,\mathbb{Z}\right)$
is generated by $\left[\omega_{\Sigma}\right]$ then $\Gamma_{L}$
is trivial. It is in particular the case for $\Sigma=\mathbb{C}P^{n}$.\end{lem}
\begin{proof}
In both cases we compute the Euler class $e_{\Gamma_{L}}$; if $H^{2}\left(L,\mathbb{Z}\right)=0$
then it is trivially zero. In the other case, let us denote by $P$
the circle bundle over $\Sigma$ of same radius as $\Gamma_{L}$,
such that $\Gamma_{L}=\iota^{*}P$ where $\iota:L\hookrightarrow\Sigma$
is the inclusion. By naturality of the Euler class, $e_{\Gamma_{L}}=\iota^{*}e_{P}$,
but $e_{P}$ is collinear to $\left[\omega_{\Sigma}\right]$, and
$\iota^{*}\omega_{\Sigma}=0$ since $L$ is a Lagrangian submanifold.
\end{proof}
The combination of those two points implies the theorem~\ref{thm:main result version CPn}.

\subsection{Connected sums}
\begin{lem}
\label{lem:free product}Let $\left(G_{i}\right)_{i\in I}$ be a finite
collection of groups, at least one being infinite and another non-trivial.
Let $g\in *_{i\in I}G_{i}\setminus\left\{ e\right\} $, where
$e$ refers to the identity. Then its centraliser $Z\left(g\right)$
is not of finite index.\end{lem}
\begin{proof}
To simplify the notations we will assume that $I=\left\{ 1,2\right\} $
with $G_{1}$ infinite and $G_{2}$ non-trivial. It is clear that
any element in $\left(G_{1}*G_{2}\right)\setminus\left\{ e\right\} $
can be uniquely written as a product of non-trivial elements of $G_{1}$
and $G_{2}$ alternately. This fact is the basis of a nice sub-lemma:
\begin{lem}
Let $y\in G_{i}\setminus\left\{ e\right\} $ and $x\in\left(G_{1}*G_{2}\right)\setminus G_{i}$.
Then $x$ and $y$ do not commute.\end{lem}
\begin{proof}
If $x$ is in the other group than $y$, obviously they do not commute.
Let us then write $x=\prod\limits _{k=1}^{n}x_{k}$ as a product of
non-trivial elements of $G_{1}$ and $G_{2}$ alternately. Since $x$
is in neither $G_{1}$ nor $G_{2}$, $k>1$. Therefore, if we write
$xy$ in the same fashion, its leftmost factor is still $x_{1}$,
even after simplification. Besides, $x_{2}$ is in the other group:
to have $yx=xy$ we therefore need $x_{1}$ to be in $G_{1}\setminus\left\{ e,y^{-1}\right\} $.
But now the leftmost factor of $yx$ is $\left(yx_{1}\right)$, which
must be equal to the leftmost factor of $xy$, that is to say $x_{1}$.
Since $y\neq e$, this is impossible.
\end{proof}
We can now use this result two ways: first, assume that $g\in G_{i}$.
Then we have that its centraliser lie in $G_{i}$. We can pick some
$h=h_{1}h_{2}$, where each $h_{i}$ is some non-trivial element of
$G_{i}$. Then for $n\in\mathbb{N}$, each $h^{n}Z\left(g\right)$
is distinct, otherwise it would imply that $h^{k}\in Z\left(g\right)$
for some $k\in\mathbb{N}$.

Now if $g\in\left(G_{1}*G_{2}\right)\setminus\left(G_{1}\cup G_{2}\right)$,
we know that $Z\left(g\right)\cap G_{i}=\left\{ e\right\} $ for $i=1,2$.
In particular we see that taking a non-repeating sequence $\left(h_{n}\right)_{n\in\mathbb{N}}$
in $G_{1}$ gives infinitely many distinct classes $h_{n}Z\left(g\right)$.

In either case, the index of $Z\left(g\right)$ is infinite.\end{proof}
\begin{rem}
On the other hand, it is reasonably straightforward to check that
in $\nicefrac{\mathbb{Z}}{2\mathbb{Z}}*\nicefrac{\mathbb{Z}}{2\mathbb{Z}}=\left\langle u,v\mid u^{2},v^{2}\right\rangle $,
$Z\left(uv\right)=\left\langle uv\right\rangle =\left\{ \left(uv\right)^{k},k\in\mathbb{Z}\right\} $
is of finite index.
\end{rem}

Combining this lemma~\ref{lem:free product} with the our main result
as stated in the theorem~\ref{thm:main result version CPn}, we obtain
this corollary: 
\begin{cor}
Let $L$ be a compact, orientable manifold such that all the odd-numbered
cohomology groups of its universal cover $\tilde{L}$ vanish. Assume
that its fundamental group is the free product of a non-trivial group
and an infinite group, and either:
\begin{enumerate}
\item $\forall i\in I,H^{2}\left(L_{i},\mathbb{Z}\right)=0$ or
\item $H^{2}\left(\Sigma,\mathbb{Z}\right)$ is generated by $\left[\omega_{\Sigma}\right]$.
\end{enumerate}
Then $L$ cannot be embedded in $\Sigma$ as a monotone Lagrangian
submanifold.\end{cor}
\begin{rem}
Let $G_{1}$ and $G_{2}$ be two non-trivial groups. Then for $i\in\left\{ 1,2\right\} $,
there exists some non-trivial $g_{i}\in G_{i}$, and $\left\{ \left(g_{1}g_{2}\right)^{n},\allowbreak n\in\mathbb{N}\right\} $
clearly is an infinite subset of $G_{1}*G_{2}$.

Hence, it suffices for the fundamental group to be the free product
of three non-trivial groups.
\end{rem}

We now prove corollary~\ref{cor:corollaire prinicpal}: 
\begin{cor*}
Let $\left(L_{i}\right)_{i\in I}$ be a finite collection of compact,
orientable, $2k+1$-di\-men\-sio\-nal manifolds such that all the
odd-numbered cohomology groups of each universal cover $\tilde{L}_{i}$
vanish. Assume that either:
\begin{enumerate}
\item $\forall i\in I,H^{2}\left(L_{i},\mathbb{Z}\right)=0$ with $k>1$,
or
\item $H^{2}\left(\Sigma,\mathbb{Z}\right)$ is generated by $\left[\omega_{\Sigma}\right]$.
\end{enumerate}
Then there is no Lagrangian monotone embedding of the connected sum
$\sharp_{i\in I}L_{i}$ in $\Sigma$.\end{cor*}
\begin{proof}
Since our $\left(L_{i}\right)_{i\in I}$ are compact manifolds, so
are their universal covers whenever the fundamental group is finite.
Yet their $2k+1$-cohomology groups vanish, so it is impossible.

Then, using the Mayer-Vietoris sequence, it is easy to see that the
odd-numbered cohomology groups of each universal cover vanish also
for the connected sum $\sharp_{i\in I}L_{i}$. The same reasoning
shows that the assumption $1$ is stable through connected sums.
\end{proof}

Corollary~\ref{cor:application S1xS2k CP2k+1} is then a straightforward
application.

%\address{
\vfill
\setlength{\parindent}{0cm}
Simon Schatz\\
IRMA, UMR 7501\\
7, rue René-Descartes\\
67084 Strasbourg Cedex\\
France
%}

\newpage{}

\bibliographystyle{alpha}
\nocite{*}
\bibliography{articlebis}

\end{document}